\documentclass[12pt,reqno]{amsart}
\usepackage{amsmath,amssymb,amsfonts,amsthm}
\usepackage[left=3cm, right=3cm, top=2.8cm, bottom=2.8cm]{geometry}
\usepackage[utf8]{inputenc}
\usepackage[T1]{fontenc}
\usepackage{color,xcolor}
\usepackage{multirow}
\usepackage{graphicx}
\usepackage{hyperref}
\usepackage{refcheck}
\usepackage{graphicx}
\usepackage{calc}%

\newtheorem{theorem}{Theorem}
\newtheorem{lemma}[theorem]{Lemma}
\newtheorem{proposition}[theorem]{Proposition}
\newtheorem{corollary}[theorem]{Corollary}
\newtheorem{definition}[theorem]{Definition}
\newtheorem{remark}[theorem]{Remark}

\theoremstyle{plain}

\begin{document}

\title[On Continuously Differentiable Vector-Valued Functions]{On Continuously Differentiable Vector-Valued Functions of Non-Integer Order}


\author[P. M. Carvalho-Neto]{Paulo M. de Carvalho-Neto}
\address[Paulo M. de Carvalho Neto]{Departamento de Matem\'atica, Centro de Ci\^{e}ncias F\'{i}sicas e Matem\'aticas, Universidade Federal de Santa Catarina, Florian\'{o}polis - SC, Brazil}
\email[]{paulo.carvalho@ufsc.br}
\author[R. Fehlberg J\'{u}nior]{Renato Fehlberg J\'{u}nior}
\address[Renato Fehlberg J\'{u}nior]{Departamento de Matem\'atica, Universidade Federal do Esp\'{i}rito Santo, Vit\'{o}ria - ES, Brazil}
\email[]{renato.fehlberg@ufes.br}






\begin{abstract}
The function spaces of continuously differentiable functions are extensively studied and appear in various mathematical settings. In this context, we investigate the spaces of continuously fractional differentiable functions of order $\alpha>0$, considering both the Riemann-Liouville and Caputo fractional derivatives. We explore several fundamental properties of these spaces and, inspired by a result of Hardy and Littlewood, we compare them with the space of Hölder continuous functions. Our main objective is to establish a rigorous theoretical framework to support the study and further advancement of this subject.
\end{abstract}

\maketitle

\section{Introduction}

Since its inception, fractional calculus has attracted the attention of many mathematicians and researchers from various scientific fields, due to both its rich theory, offering nuances distinct from classical calculus, and its wide range of applications (for a historical overview of fractional calculus, see \cite{OlSp1,Ro1}, and for classical results, refer to \cite{SaKiMa1} and the references therein).

A crucial contribution to the field was made by Hardy and Littlewood, who were among the first to prove fundamental results regarding the continuity of the Riemann-Liouville fractional integral of order $\alpha>0$; see \cite{HaLi1} for details. Their studies inspired numerous mathematicians to further investigate this operator, leading to a vast body of research (see, for instance, a list of references in \cite{CarFe0}).

Although many researchers study various topics in fractional calculus, several fundamental questions in the theory still require deeper investigation. In recent years, the authors have worked on addressing some of these gaps. For instance, a recent series of papers examined the Riemann-Liouville fractional integral as an operator on Bochner-Lebesgue spaces (see \cite{CarFe0, CarFe2, CarFe3, CarFe4}).

A natural continuation of these studies is the connection between the continuity of fractional derivatives of continuous functions and differentiable spaces and its relationship with Hölder spaces. The regularity of fractional derivatives is closely tied to these spaces, and various works have discussed this relationship (see for instance \cite{AlBaMk1, Bo1, KaGi1, KaSh1, RoStSa1}).

This work is organized as follows: in Section \ref{sec2}, we provide the necessary background and definitions. In Section \ref{sec3}, we introduce the function spaces of fractional continuously differentiable functions and study their interrelations. Finally, in Section \ref{sec4}, we discuss their regularity, connections with Hölder spaces, and show that these spaces form Banach algebras.

\section{Preliminary Concepts}\label{sec2}

In this section, we introduce the definitions of the key objects used throughout this manuscript and establish some of their fundamental properties. We assume that $t_0 < t_1$ are fixed real numbers and that $X$ is a Banach space.

An in-depth discussion of the Bochner-Lebesgue spaces $L^p(t_0,t_1;X)$ can be found in \cite{CarFe0} and the references therein. For further details on the space of continuously differentiable functions $C^n([t_0,t_1];X)$ and the Bochner-Sobolev spaces $W^{1,p}(t_0,t_1;X)$, we refer to \cite{CaHa1}.

\begin{definition}
Consider $\alpha \in (0,\infty)$ and a function $f:[t_0,t_1] \to X$.
\begin{itemize}
\item[(i)] The Riemann-Liouville (RL) fractional integral of order $\alpha$ at $t_0$ is defined as
\begin{equation}\label{fracinit}
J_{t_0,t}^\alpha f(t) := \frac{1}{\Gamma(\alpha)} \int_{t_0}^{t} (t-s)^{\alpha-1} f(s) \,ds,
\end{equation}
for every $t \in [t_0,t_1]$ where the integral \eqref{fracinit} exists. Here, $\Gamma$ denotes the Euler gamma function.\vspace*{0.2cm}
\item[(ii)] The Riemann-Liouville (RL) fractional derivative of order $\alpha$ at $t_0$ is given by
\begin{equation}\label{202502231625}
D_{t_0,t}^\alpha f(t) := \frac{d^{[\alpha]}}{dt^{[\alpha]}} \left[J_{t_0,t}^{[\alpha]-\alpha} f(t)\right],
\end{equation}
for every $t \in [t_0,t_1]$ where the right-hand side of \eqref{202502231625} exists. Here, and throughout this work, $[\alpha]$ denotes the smallest integer greater than or equal to $\alpha$. The derivative $d^{[\alpha]}/dt^{[\alpha]}$ is understood in the weak sense. \vspace*{0.2cm}
\item[(iii)] The Caputo fractional derivative of order $\alpha$ at $t_0$ is given by
\begin{equation}\label{202502231626}
cD_{t_0,t}^\alpha f(t) := D_{t_0,t}^\alpha \left[ f(t) - \sum_{j=0}^{[\alpha]-1} \frac{f^{(j)}(t_0)}{j!} (t-t_0)^j \right],
\end{equation}
for every $t \in [t_0,t_1]$ where the right-hand side of \eqref{202502231626} exists.
\end{itemize}
\end{definition}

\begin{remark}\label{eqre01} Let us begin by pointing out some classical facts concerning the fractional operators introduced above.
\begin{itemize}
\item[(i)] For $p\in[0,\infty]$, $J_{t_0,t}^\alpha:L^p(t_0,t_1;X)\rightarrow L^p(t_0,t_1;X)$ is a bounded operator. See \cite[Theorem 3.1]{CarFe0}.\vspace*{0.2cm}
\item[(ii)] Assume that $p=\infty$ and $\alpha>0$ or that $p\in(1,\infty)$ and $\alpha\in(1/p,\infty)$. Then, if $f \in L^p(t_0,t_1;X)$, we have $J_{t_0,t}^\alpha f \in C([t_0,t_1];X)$. For details, see {\color{blue}\cite[Theorems 7]{CarFe4}}.\vspace*{0.2cm}
\item[(iii)]For any $f\in L^1(t_0,t_1;X)$, we have that
\begin{equation*}\lim_{\alpha\rightarrow0^+}{\big\|J_{t_0,t}^\alpha f-f\big\|_{L^1(t_0,t_1;X)}}=0.\end{equation*}
Therefore, we define $J_{t_0,t}^0 f(t):=f(t),$ for every $f\in L^1(t_0,t_1;X)$ (cf. \cite[Theorem 3.10]{CarFe0}). \vspace*{0.2cm}
\item[(iv)]  As a consequence of item $(iii)$, if $\alpha \in \mathbb{N}$, both $D_{t_0,t}^\alpha$ and $cD_{t_0,t}^\alpha$ correspond to the $\alpha$-th weak derivative of a function.
\end{itemize}
\end{remark}

There are several other important properties related to these operators. Therefore, from now on, to keep this manuscript as concise and self-contained as possible, we only prove results for which we are unaware of any satisfactory existing proof. Otherwise, we provide full citations, allowing readers to locate the reference for the stated fact.

\begin{proposition}\label{auxkilbas1} Let $\alpha,\beta\in(0,\infty)$ be fixed real numbers.
\begin{enumerate}
\item[(i)] (\cite[Theorem 3.15]{CarFe0}) For any $f\in L^1(t_0,t_1;X)$ it holds that
$$J_{t_0,t}^{\alpha+\beta} f(t)=J_{t_0,t}^{\alpha}\left[J_{t_0,t}^{\beta} f(t)\right],$$
for almost every $t\in [t_0,t_1]$. \vspace*{0.2cm}
\item[(ii)] Let $m\in\mathbb{N}^*:=\{1,2,\ldots\}$. If $f\in W^{m,1}\big([t_0,t_1];X\big)$ then
$$J_{t_0,t}^{\alpha} f(t)=J_{t_0,t}^{\alpha+m} f^{(m)}(t)+\sum_{j=0}^{m-1}\left[\dfrac{(t-t_0)^{\alpha+j}f^{(j)}(t_0)}{\Gamma(\alpha+j+1)}\right],$$
for almost every $t\in [t_0,t_1]$. \vspace*{0.2cm}
\item[(iii)] Let $m\in\mathbb{N}^*:=\{1,2,\ldots\}$. Assume that $f\in W^{m,1}\big([t_0,t_1];X\big)$, with $f^{(j)}(t_0)=0$ for every $j\in\{0,1,\ldots,m-1\}$. Then,
$$\dfrac{d^{m}}{dt^{m}}\bigg[J_{t_0,t}^{\alpha} f(t)\bigg]=J_{t_0,t}^{\alpha} f^{(m)}(t),$$
for almost every $t\in [t_0,t_1]$. \vspace*{0.2cm}
\end{enumerate}
\end{proposition}

\begin{proof} To prove item $(ii)$, note that item $(i)$ guarantees the identity
\begin{equation*}
J_{t_0,t}^{\alpha+m} f^{(m)}(t) = J_{t_0,t}^{\alpha+m-1} \Big[J_{t_0,t}^{1} f^{(m)}(t)\Big]
= J_{t_0,t}^{\alpha+m-1} \Big[f^{(m-1)}(t) - f^{(m-1)}(t_0)\Big],
\end{equation*}
for almost every $t \in [t_0,t_1]$. Applying this iteratively, we obtain the desired result.

Finally, item $(iii)$ follows directly from item $(ii)$.
\end{proof}

We conclude this section by presenting a final result that establishes crucial relationships between the fractional operators, which are used recursively throughout this manuscript.

\begin{theorem}\label{auxkilbas2} Consider $m\in\mathbb{N}^*$, $\alpha\in(m-1,m)$, and assume that $f\in C^{m-1}\big([t_0,t_1];X\big)$.
\begin{enumerate}
    \item[(i)] (\cite[Proposition 4.6]{CarFe01}) If $m=1$, then $J_{t_0,t}^{1-\alpha} f(t)\big|_{t=t_0}=0$.\vspace*{0.2cm}
    \item[(ii)] For $m>1$, if $J_{t_0,t}^{m-\alpha} f\in C^{m-1}\big([t_0,t_1];X\big)$, then
    $$f^{(j)}(t_0)= 0\qquad\text{and}\qquad\dfrac{d^{j}}{dt^{j}}\Big[J_{t_0,t}^{m-\alpha} f(t)\Big]\bigg|_{t=t_0}=0,$$
    for every $j\in\{0,\dots,m-2\}$. Moreover,
    $$\dfrac{d^{m-1}}{dt^{m-1}}\Big[J_{t_0,t}^{m-\alpha} f(t)\Big]\bigg|_{t=t_0}=0.$$
   \item[(iii)] If $J_{t_0,t}^{m-\alpha} f\in C^{m}\big([t_0,t_1];X\big)$, then
    $$J_{t_0,t}^{\alpha}\Big[D_{t_0,t}^\alpha f(t)\Big]=f(t),\quad \text{for every } t\in [t_0,t_1].$$
\end{enumerate}
\end{theorem}

\begin{proof} To prove item $(ii)$, first assume that $m=2$. Notice that item $(i)$ allows us to conclude that
$$
J_{t_0,t}^{2-\alpha} f(t)\Big|_{t=t_0} = J_{t_0,t}^{1-(\alpha-1)} f(t)\Big|_{t=t_0} = 0.
$$

Now, item $(ii)$ of Proposition \ref{auxkilbas1} guarantees that
\begin{equation}\label{202502251200}
  J_{t_0,t}^{2-\alpha} f(t) = J_{t_0,t}^{2-\alpha+1} f^{\prime}(t) + \frac{(t-t_0)^{2-\alpha} f(t_0)}{\Gamma(3-\alpha)},
\end{equation}

for every $t \in [t_0,t_1]$. Differentiating both sides of \eqref{202502251200}, we obtain
\begin{equation*}
\frac{d}{dt} \Big[J_{t_0,t}^{2-\alpha} f(t)\Big] - J_{t_0,t}^{2-\alpha} f^{\prime}(t) = \frac{(t-t_0)^{1-\alpha} f(t_0)}{\Gamma(2-\alpha)},
\end{equation*}
for every $t \in (t_0,t_1]$. Since the left-hand side of the above identity is a bounded function in $[t_0,t_1]$ (see item $(ii)$ of Remark \ref{eqre01}) and $1-\alpha < 0$, we conclude that $f(t_0) = 0$.

Using this result, along with \eqref{202502251200} and item $(i)$, and noting that $f^\prime(t)$ is continuous in $[t_0,t_1]$, we obtain
$$
\frac{d}{dt}\Big[J_{t_0,t}^{2-\alpha} f(t)\Big]\bigg|_{t=t_0} = J_{t_0,t}^{2-\alpha} f^{\prime}(t)\Big|_{t=t_0} = J_{t_0,t}^{1-(\alpha-1)} f^{\prime}(t)\Big|_{t=t_0} = 0.
$$
Following the same reasoning, we obtain the general case.

Finally, to prove item $(iii)$, observe that items $(i)$ and $(ii)$ of this theorem ensure, for every $j\in\{0,\cdots,m-1\}$, that
$$\dfrac{d^{(j)}}{dt^{(j)}}\Big[J_{t_0,t}^{m-\alpha} f(t)\Big]\bigg|_{t=t_0}=0.$$

Therefore, items $(i)$ and $(iii)$ of Proposition \ref{auxkilbas1} guarantee the identity
$$J_{t_0,t}^{\alpha}\bigg[D_{t_0,t}^{\alpha} f(t)\bigg]=J_{t_0,t}^{\alpha}\bigg\{\dfrac{d^{m}}{dt^{m}}\Big[J_{t_0,t}^{m-\alpha} f(t)\Big]\bigg\}=\dfrac{d^{m}}{dt^{m}}\bigg\{J_{t_0,t}^{\alpha}\Big[J_{t_0,t}^{m-\alpha}f(t)\Big]\bigg\}=f(t),$$
for every $t\in[t_0,t_1]$.

\end{proof}


\section{The spaces of fractional continuously differentiable functions}\label{sec3}

In this section, we introduce the spaces of functions with continuous fractional derivatives, denoted by $RL^\alpha([t_0,t_1];X)$ and $C^\alpha([t_0,t_1];X)$. We establish some fundamental properties of these spaces and explore the relationships between them.

\begin{definition} Let $\alpha\in(0,\infty)$ and consider the following function spaces:
\begin{itemize}
\item[(i)] the RL $\alpha-$times continuously differentiable functions in $[t_0,t_1]$ is the set
$$RL^\alpha\big([t_0,t_1],X\big):=\Big\{f\in C^{[\alpha]-1}([t_0,t_1],X): D^\alpha_{t_0,t}f\in C^0([t_0,t_1],X)\Big\}.$$
\item[(ii)] the Caputo $\alpha-$times continuously differentiable functions in $[t_0,t_1]$ is the set
$$C^\alpha\big([t_0,t_1],X\big):=\Big\{f\in C^{[\alpha]-1}([t_0,t_1],X): cD^\alpha_{t_0,t}f\in C^0([t_0,t_1],X)\Big\}.$$
\end{itemize}
\end{definition}

\begin{remark}\label{202502272337}
From item $(iv)$ of Remark \ref{eqre01}, it follows that for $\alpha\in\mathbb{N}^*$, the spaces $C^\alpha\big([t_0,t_1],X\big)$ and $RL^\alpha\big([t_0,t_1],X\big)$ coincide with the classical space of $\alpha$-times continuously differentiable functions.\vspace*{0.2cm}


\end{remark}

Besides being natural generalizations of the space of continuously differentiable functions, the spaces defined above are more subtle than they may appear. Below, we present our main results concerning these fractional differentiability spaces.

\begin{proposition}\label{nulo}
  Consider $m\in\mathbb{N}^*$ and $\alpha\in(m-1,m)$. If $f\in RL^\alpha\big([t_0,t_1],\mathbb{R}\big)$, then $f^{(j)}(t_0)= 0$ for every $j\in\{0,\cdots,m-1\}$.
\end{proposition}

\begin{proof}
Assume initially that $m=1$. Hence, item $(iii)$ of Theorem \ref{auxkilbas2} guarantees
$$J_{t_0,t}^\alpha\Big[D_{t_0,t}^\alpha f(t)\Big]=f(t),$$
for every $t\in[t_0,t_1]$. However, since $D_{t_0,t}^{\alpha} f\in C^{0}\big([t_0,t_1],X\big)$, we have from item $(i)$ of Theorem \ref{auxkilbas2}
$$0=J_{t_0,t}^{1-(1-\alpha)}\Big[D_{t_0,t}^\alpha f(t)\Big]\bigg|_{t=t_0}=J_{t_0,t}^\alpha\Big[D_{t_0,t}^\alpha f(t)\Big]\bigg|_{t=t_0}=f(t_0).$$

If $m>1$ and $\alpha\in(m-1,m)$, since $f\in RL^\alpha\big([t_0,t_1],\mathbb{R}\big)$, item $(ii)$ of Theorem \ref{auxkilbas2} ensures that
$$f^{(j)}(t_0)=0,$$
for every $j\in\{0,\cdots,m-2\}$. Now, by item $(iii)$ of Proposition \ref{auxkilbas1} we obtain
\begin{equation}\label{202502251724}
  D_{t_0,t}^{\alpha}f(t)=\dfrac{d^{(m)}}{dt^{(m)}}\bigg[J_{t_0,t}^{m-\alpha}f(t)\bigg]
=\dfrac{d}{dt}\bigg[J_{t_0,t}^{m-\alpha}f^{(m-1)}(t)\bigg]=D_{t_0,t}^{1-m+\alpha}f^{(m-1)}(t),
\end{equation}
for every $t\in[t_0,t_1]$.

Since $D_{t_0,t}^{\alpha}f\in C^0\big([t_0,t_1],X\big)$, we deduce that $D_{t_0,t}^{1-m+\alpha}f^{(m-1)}\in C^0\big([t_0,t_1],X\big)$. Hence, the fact that $f^{(m-1)}\in C^0\big([t_0,t_1],X\big)$ allow us to conclude, by considering \eqref{202502251724} and item $(iii)$ of Theorem \ref{auxkilbas2}, that
$$J_{t_0,t}^{1-m+\alpha}\Big[D_{t_0,t}^{\alpha}f(t)\Big]=J_{t_0,t}^{1-m+\alpha}\Big[D_{t_0,t}^{1-m+\alpha}f^{(m-1)}(t)\Big]=f^{(m-1)}(t),$$
for every $t\in[t_0,t_1]$. Like it was done in the first part of this proof,
$$0=J_{t_0,t}^{1-m+\alpha}\Big[D_{t_0,t}^{1-m+\alpha}f^{(m-1)}(t)\Big]\bigg|_{t=t_0}=f^{(m-1)}(t_0).$$
\end{proof}

\begin{theorem}\label{auxaux1} Let $\alpha\in(0,\infty)$. If we consider
\begin{equation*}\|f\|_{RL^{\alpha}([t_0,t_1];X)}:=\|f\|_{C^{[\alpha]-1}([t_0,t_1];X)}+\|D^\alpha_{t_0,t}f\|_{C^{0}([t_0,t_1];X)},\end{equation*}
then $\big(RL^{\alpha}([t_0,t_1];X),\|\cdot\|_{RL^{\alpha}([t_0,t_1];X)}\big)$ is a Banach space.
\end{theorem}
\begin{proof} The fact that $RL^{\alpha}([t_0,t_1];X)$ is a vector space and that $\|f\|_{RL^{\alpha}([t_0,t_1];X)}$ defines a norm follows directly from the definition. Since the case $\alpha\in\mathbb{N}^*$ is well known, we assume $\alpha\in(0,\infty)\setminus\mathbb{N}^*$. Consider a Cauchy sequence $\{f_n\}_{n=1}^\infty\subset RL^{\alpha}\big([t_0,t_1];X\big)$. Then there exists $f\in C^{[\alpha]-1}\big([t_0,t_1];X\big)$ such that
\begin{equation}\label{auxhelp1}
\|f_n-f\|_{C^{[\alpha]-1}([t_0,t_1];X)}\rightarrow0,\quad\text{as } n\rightarrow\infty.
\end{equation}

From \eqref{auxhelp1}, item $(i)$ of Remark \ref{eqre01}, item $(iii)$ of Proposition \ref{auxkilbas1} and Proposition \ref{nulo}, we obtain
\begin{equation*}
  \|J_{t_0,t}^{[\alpha]-\alpha} f_n - J_{t_0,t}^{[\alpha]-\alpha} f\|_{C^{[\alpha]-1}([t_0,t_1];X)}\rightarrow 0,\quad\text{as } n\rightarrow\infty.
\end{equation*}

Since $\{D^\alpha_{t_0,t}f_n\}_{n=1}^\infty$ is a Cauchy sequence in $C^{0}([t_0,t_1];X)$, there exists $g\in C^{0}([t_0,t_1];X)$ such that
\begin{equation*}
\|D^\alpha_{t_0,t}f_n-g\|_{C^{0}([t_0,t_1];X)}\rightarrow 0,\quad\text{as } n\rightarrow\infty.
\end{equation*}

Finally, due to \cite[Theorem 7.17]{Ru1} we have $J_{t_0,t}^{[\alpha]-\alpha} f\in C^{[\alpha]}([t_0,t_1];X)$ and
\begin{equation*}
g(t)= D^\alpha_{t_0,t}f(t),\quad \text{for every } t\in[t_0,t_1],
\end{equation*}
which completes the proof.

\end{proof}

As an consequence, we have:

\begin{corollary}\label{coro}
Let $\alpha\in(0,\infty)$. If we consider
\begin{equation*} \|f\|_{C^{\alpha}([t_0,t_1];X)}:=\|f\|_{C^{[\alpha]-1}([t_0,t_1];X)}+\|cD^\alpha_{t_0,t}f\|_{C^{0}([t_0,t_1];X)},\end{equation*}
then $\big(C^{\alpha}([t_0,t_1];X),\|\cdot\|_{C^{\alpha}([t_0,t_1];X)}\big)$ is a Banach space.
\end{corollary}

\begin{proof}
 Again, the fact that $C^{\alpha}([t_0,t_1];X)$ is a vector space and that $\|\cdot\|_{C^{\alpha}([t_0,t_1];X)}$ defines a norm follows directly from the definition. Since the case $\alpha \in \mathbb{N}^*$ is well known, we assume $\alpha \in (0,\infty)\setminus\mathbb{N}^*$ and consider a Cauchy sequence $\{f_n\}_{n=1}^{\infty}\subset C^{\alpha}([t_0,t_1];X)$. Define
$$
g_n(t) := f_n(t)-\sum_{j=0}^{[\alpha]-1}\frac{f_n^{(j)}(t_0)}{j!}(t-t_0)^{j}.
$$
Since $g_n^{(j)}(t_0)=0$ for every $j\in\{0,\dots,[\alpha]-1\}$, we have that $g_n \in RL^{\alpha}([t_0,t_1];X)$. By Theorem \ref{auxaux1}, there exists $g\in RL^{\alpha}([t_0,t_1];X)$ such that
$$
\|g_n-g\|_{RL^{\alpha}([t_0,t_1];X)}\rightarrow 0, \quad \text{as } n\rightarrow\infty.
$$

On the other hand, since $\{f_n\}_{n=1}^\infty$ is a Cauchy sequence in $C^{\alpha}([t_0,t_1];X)$, there exists $f\in C^{[\alpha]-1}([t_0,t_1];X)$ satisfying
$$
\|f_n-f\|_{C^{[\alpha]-1}([t_0,t_1];X)}\rightarrow 0, \quad\text{as } n\rightarrow\infty.
$$

Due to the uniform convergence, we must have
$$
g(t) = f(t)-\sum_{j=0}^{[\alpha]-1}\frac{f^{(j)}(t_0)}{j!}(t-t_0)^{j},
$$
for every $t\in[t_0,t_1]$. Therefore, since $D^\alpha_{t_0,t}g_n=cD^\alpha_{t_0,t}f_n$ and $D^\alpha_{t_0,t}g=cD^\alpha_{t_0,t}f$, we obtain
$$
\|cD^\alpha_{t_0,t}f_n - cD^\alpha_{t_0,t}f\|_{C^{0}([t_0,t_1];X)}\rightarrow 0, \quad\text{as } n\rightarrow\infty,
$$
completing the proof.
\end{proof}

We now turn our attention to the relationship between the spaces $RL^{\alpha}\big([t_0,t_1];X\big)$ and $C^{\alpha}\big([t_0,t_1];X\big)$.

\begin{remark}\label{eqre02111}
Assume that $\alpha\in(0,\infty)$.
  \begin{itemize}
    \item[(i)] If $\alpha\not\in\mathbb{N}$ and $f\in RL^{\alpha}\big([t_0,t_1];X\big)$, then
    $$
    \|f\|_{RL^{\alpha}([t_0,t_1];X)} = \|f\|_{C^{\alpha}([t_0,t_1];X)}.
    $$
    Consequently, we obtain the following characterization:
    \begin{equation*}
    RL^{\alpha}([t_0,t_1];X) = \Big\{ f\in C^{\alpha}([t_0,t_1];X) : f^{(j)}(t_0) = 0, \,\, \forall j\in\{0,\dots,[\alpha]-1\} \Big\}.
    \end{equation*}

    \item[(ii)] If $\alpha>1$ and $f\in RL^{\alpha}([t_0,t_1];X)$, then for every $j\in\{1,\dots,[\alpha]-1\}$, we have
    \[
    D_{t_0,t}^\alpha f(t) = D_{t_0,t}^{\alpha-j} f^{(j)}(t) \in C^0([t_0,t_1];X).
    \]
    Consequently, $f^{(j)}\in RL^{\alpha-j}([t_0,t_1];X)$, for every $j\in\{1,\dots,[\alpha]-1\}$.
  \end{itemize}
\end{remark}

Therefore, we have the following

\begin{proposition} For $\alpha>0$, it holds that $RL^{\alpha}([t_0,t_1];X)$ is continuously embedded in $C^\alpha([t_0,t_1];X)$. Moreover, if $\alpha\in\mathbb{N}$, then $RL^{\alpha}([t_0,t_1];X) = C^\alpha([t_0,t_1];X)$, whereas if $\alpha\not\in\mathbb{N}$, then $RL^{\alpha}([t_0,t_1];X) \subsetneq C^\alpha([t_0,t_1];X)$.
\end{proposition}

\begin{proof} For $\alpha\in\mathbb{N}$, the equality (and therefore the continuous inclusion) follows from item $(iv)$ of Remark \ref{eqre01} and Remark \ref{202502272337}.

If $\alpha\not\in\mathbb{N}$, then the continuous embedding follows direct from item $(i)$ from Remark \ref{eqre02111}.  Finally, the strictness for the case $\alpha\not\in\mathbb{N}$ follows from the fact that the Mittag-Leffler function $E_\alpha\big((t-t_0)^\alpha\big)$, defined by
$$ E_\alpha\big((t-t_0)^\alpha\big):=\sum_{j=1}^\infty\dfrac{(t-t_0)^{\alpha j}}{\Gamma(\alpha j+1)},$$
as well as nonzero constant functions, are examples of functions that belong to the space $C^\alpha\big([t_0,t_1],\mathbb{R}\big)$ but do not belong to $RL^\alpha\big([t_0,t_1],\mathbb{R}\big)$. Indeed, observe that
$$cD^\alpha_{t_0,t}c=0\qquad\textrm{while}\qquad D^\alpha_{t_0,t}c=\dfrac{(t-t_0)^{-\alpha}c}{\Gamma(1-\alpha)},$$
and also (see \cite{HaMaSa1} for details)
$$cD^\alpha_{t_0,t}E_\alpha\big(t^\alpha\big)=E_\alpha\big((t-t_0)^\alpha\big)\quad\textrm{while}\quad D^\alpha_{t_0,t}E_{\alpha}\big((t-t_0)^\alpha\big)=E_{\alpha}\big((t-t_0)^\alpha\big)-\dfrac{(t-t_0)^{-\alpha}}{\Gamma(1-\alpha)}.$$
\end{proof}

We now address the relationships between RL $\alpha$-times and $\beta$-times continuously differentiable functions. For this goal, we first establish an identity connecting these operators.

\begin{lemma}\label{202502271758} For $\alpha,\beta\in [0,\infty)$, with $\alpha<\beta$, if $f\in RL^\beta([t_0,t_1];X)$, we have
$$D_{t_0,t}^\alpha f(t)=\left\{\begin{array}{ll}J_{t_0,t}^{\beta-\alpha}\Big[f^{(\beta)}(t)\Big]+\displaystyle\sum_{j=\alpha}^{\beta-1}\left[\dfrac{(t-t_0)^{-\alpha+j}f^{(j)}(t_0)}{\Gamma(-\alpha+j+1)}\right],&\alpha,\beta\in\mathbb{N},\vspace*{0.2cm}\\
J_{t_0,t}^{\beta-\alpha}\Big[f^{(\beta)}(t)\Big]+\displaystyle\sum_{j=0}^{\beta-1}\left[\dfrac{(t-t_0)^{-\alpha+j}f^{(j)}(t_0)}{\Gamma(-\alpha+j+1)}\right],&\alpha\not\in\mathbb{N}\textrm{ and }\beta\in\mathbb{N},\vspace*{0.2cm}\\
J_{t_0,t}^{\beta-\alpha}\Big[D_{t_0,t}^\beta f(t)\Big],&\beta\not\in\mathbb{N}.\end{array}\right.$$
for every $t\in[t_0,t_1]$, except in the second case, where the identity holds only for $t\in(t_0,t_1]$.
\end{lemma}

\begin{proof} Let us split this proof into two cases.

(i) Assume that $\beta \in\mathbb{N}$. Then, item $(ii)$ of Proposition \ref{auxkilbas1} ensures that
\begin{equation}\label{202502271708}
J_{t_0,t}^{(\beta-\alpha)+\beta}\big(f^{(\beta)}(t)\big) = J_{t_0,t}^{\beta-\alpha} f(t) - \sum_{j=0}^{\beta-1} \frac{(t-t_0)^{\beta-\alpha+j} f^{(j)}(t_0)}{\Gamma(\beta-\alpha+j+1)},
\end{equation}
for every $t\in[t_0,t_1]$.

If $\alpha\in\mathbb{N}$, by differentiating $\beta$ times both sides of \eqref{202502271708}, we obtain
$$
J_{t_0,t}^{\beta-\alpha}\big(f^{(\beta)}(t)\big) = f^{(\alpha)}(t) - \sum_{j=\alpha}^{\beta-1} \frac{(t-t_0)^{-\alpha+j} f^{(j)}(t_0)}{\Gamma(-\alpha+j+1)},
$$
for every $t\in[t_0,t_1]$.

On the other hand, if $\alpha\not\in\mathbb{N}$, by differentiating $\beta$ times both sides of \eqref{202502271708}, we obtain
\begin{equation}\label{202502271702}
J_{t_0,t}^{\beta-\alpha}\big(f^{(\beta)}(t)\big) = \frac{d^\beta}{dt^\beta} \left(J_{t_0,t}^{\beta-\alpha} f(t)\right) - \sum_{j=0}^{\beta-1} \frac{(t-t_0)^{-\alpha+j} f^{(j)}(t_0)}{\Gamma(-\alpha+j+1)},
\end{equation}
for every $t\in(t_0,t_1]$.

Since we have
$$
\frac{d^\beta}{dt^\beta} \left(J_{t_0,t}^{\beta-\alpha} f(t)\right) = \frac{d^{[\alpha]}}{dt^{[\alpha]}} \left\{ \frac{d^{\beta-[\alpha]}}{dt^{\beta-[\alpha]}} \left[J_{t_0,t}^{\beta-[\alpha]} \Big(J_{t_0,t}^{[\alpha]-\alpha} f(t)\Big) \right] \right\} = D_{t_0,t}^\alpha f(t),
$$
for every $t\in[t_0,t_1]$, it follows from \eqref{202502271702} that
$$
J_{t_0,t}^{\beta-\alpha}\big(f^{(\beta)}(t)\big) = D_{t_0,t}^\alpha f(t) - \sum_{j=0}^{\beta-1} \frac{(t-t_0)^{-\alpha+j} f^{(j)}(t_0)}{\Gamma(-\alpha+j+1)},
$$
for every $t\in(t_0,t_1]$.

(ii) Assume that $\beta \not\in\mathbb{N}$. By item $(iii)$ of Proposition \ref{auxkilbas1} and items $(i)$ and $(ii)$ of Theorem \ref{auxkilbas2}, we obtain
$$
J_{t_0,t}^{\beta-\alpha}\Big(D_{t_0,t}^\beta f(t)\Big) = \frac{d^{[\beta]}}{dt^{[\beta]}}\left[J_{t_0,t}^{\beta-\alpha}\Big(J_{t_0,t}^{[\beta]-\beta} f(t)\Big)\right],
$$
for every $t\in[t_0,t_1]$.

Finally, using the semigroup property of the RL fractional integral, we get
$$
J_{t_0,t}^{\beta-\alpha}\Big(D_{t_0,t}^\beta f(t)\Big) = \frac{d^{[\beta]}}{dt^{[\beta]}}\left[J_{t_0,t}^{[\beta]-[\alpha]}\Big(J_{t_0,t}^{[\alpha]-\alpha} f(t)\Big)\right] = \frac{d^{[\alpha]}}{dt^{[\alpha]}}\Big(J_{t_0,t}^{[\alpha]-\alpha} f(t)\Big) = D_{t_0,t}^\alpha f(t),
$$
for every $t\in[t_0,t_1]$, as desired.

\end{proof}

\begin{theorem}\label{finalinclusionRL} Consider $\alpha,\beta\in [0,\infty)$ satisfying $\alpha<\beta$ and one of the following conditions:
\begin{center}
(i)\quad $\alpha\in\mathbb{N}$;\qquad or\qquad (ii)\quad $\alpha\not\in\mathbb{N}$ and $\beta\not\in\mathbb{N}$.
\end{center}
Then
\begin{equation}\label{eqeq}
  RL^{\beta}\big([t_0,t_1];X\big)\subsetneq RL^{\alpha}\big([t_0,t_1];X\big),
\end{equation}
continuously, that is, there exists $M>0$ such that
\begin{equation}\label{202502271811}\|f\|_{RL^{\alpha}([t_0,t_1];X)}\leq M\|f\|_{RL^{\beta}([t_0,t_1];X)}.\end{equation}

Moreover, if $\alpha\not\in\mathbb{N}$ and $\beta\in\mathbb{N}$, $RL^{\beta}\big([t_0,t_1];X\big)\not\subset RL^{\alpha}\big([t_0,t_1];X\big)$.
\end{theorem}

\begin{proof} If condition $(i)$ holds, \eqref{eqeq} and \eqref{202502271811} follow from the definition of the RL fractional derivative and the classical continuous embedding between spaces of continuously differentiable functions.

When condition $(ii)$ holds, consider $f\in RL^{\beta}([t_0,t_1];X)$. Then, there exists $m\in \mathbb{N}^*$ such that
$$m-1<\alpha< \beta<m,$$
or there exist $m,n\in \mathbb{N}^*$ such that
$$m-1<\alpha< m\leq n-1< \beta<n.$$

Regardless of the case, the proof of \eqref{eqeq} reduces to showing that $D_{t_0,t}^{\alpha} f(t)$ exists and is continuous in $[t_0,t_1]$. However, this follows directly from Lemma \ref{202502271758} and item $(ii)$ of Remark \ref{eqre01}.

Finally, to prove \eqref{202502271811}, recall that Lemma \ref{202502271758} also ensures the identity
$$D_{t_0,t}^\alpha f(t)=J_{t_0,t}^{\beta-\alpha}\Big[D_{t_0,t}^\beta f(t)\Big],$$
for every $t\in[t_0,t_1]$. Hence, item $(i)$ of Remark \ref{eqre01} guarantees that
\begin{multline*}
 \hspace*{1cm} \|f\|_{RL^{\alpha}([t_0,t_1];X)}= \|f\|_{C^{m-1}([t_0,t_1];X)}+\big\|J_{t_0,t}^{\beta-\alpha}\big[D_{t_0,t}^\beta f\big]\big\|_{C^{0}([t_0,t_1];X)}\\
  \leq \max\{1,M\}\|f\|_{RL^{\beta}([t_0,t_1];X)},
\end{multline*}
for some constant $M>0$.

Finally, observe that the function $f:[0,1]\rightarrow\mathbb{R}$ given by $f(t)=t^{\alpha}$ satisfies
$$D_{0,t}^\alpha f(t)=\Gamma(\alpha+1),$$
which is continuous on $[0,1]$, whereas
$$D_{0,t}^\beta f(t)=\dfrac{\Gamma(\alpha+1)}{\Gamma(\alpha-\beta+1)}t^{\alpha-\beta},$$
is not continuous on $[0,1]$. The general case follows directly from this example.

To the last assertion, notice that nonzero constant functions belong to $RL^{\beta}([t_0,t_1];X)$ but not to $RL^{\alpha}([t_0,t_1];X)$.

\end{proof}

We now address the relationships between fractional differential spaces when considering the Caputo fractional derivative.

\begin{theorem}\label{finalinclusionC} Consider $\alpha,\beta\in [0,\infty)$ satisfying $\alpha<\beta$. Then
\begin{equation*}
  C^{\beta}\big([t_0,t_1];X\big)\subsetneq C^{\alpha}\big([t_0,t_1];X\big)
\end{equation*}
continuously, that is, there exists $M>0$ such that
\begin{equation*}\|f\|_{C^{\alpha}([t_0,t_1];X)}\leq M\|f\|_{C^{\beta}([t_0,t_1];X)}.\end{equation*}
\end{theorem}

\begin{proof} Let us divide this proof into three cases.
\begin{itemize}
\item[(i)] For $\alpha,\beta\in\mathbb{N}$ or $\alpha\in\mathbb{N}$ and $\beta\not\in\mathbb{N}$, the proof follows directly from the definition.\vspace*{0.2cm}

\item[(ii)] For $\alpha\not\in\mathbb{N}$ and $\beta\in\mathbb{N}$, consider $f\in C^{\beta}([t_0,t_1];X)$. Since $[\alpha]\leq\beta$, by item $(iii)$ of Proposition \ref{auxkilbas1},
$$
J_{t_0,t}^{[\alpha]-\alpha} \left[f(t)-\sum_{j=0}^{[\alpha]-1}\frac{f^{(j)}(t_0)}{j!}(t-t_0)^{j}\right]\in C^{[\alpha]}([t_0,t_1];X).
$$
Therefore, $f\in C^{\alpha}([t_0,t_1];X)$. Moreover, item $(iii)$ of Proposition \ref{auxkilbas1} ensures that
\begin{multline*}
\qquad\qquad\|cD^\alpha_{t_0,t}f\|_{C^{0}([t_0,t_1];X)}\\\leq\left\|J_{t_0,t}^{[\alpha]-\alpha}\left[\frac{d^{[\alpha]}}{dt^{[\alpha]}}\left(f-\sum_{j=0}^{[\alpha]-1}\frac{f^{(j)}(t_0)}{j!}(\,\cdot-t_0)^{j}\right)\right]\right\|_{C^{0}([t_0,t_1];X)}.
\end{multline*}

Thus, by item $(i)$ of Remark \ref{eqre01}, we obtain
\begin{multline*}
  \qquad\qquad\|f\|_{C^{\alpha}([t_0,t_1];X)}=\|f\|_{C^{[\alpha]-1}([t_0,t_1];X)}+\|cD^\alpha_{t_0,t}f\|_{C^{0}([t_0,t_1];X)}
    \\\leq \max\{1,M\}\|f\|_{C^{\beta}([t_0,t_1];X)},
\end{multline*}
for some constant $M>0$.\vspace*{0.2cm}

\item[(iii)] If $\alpha,\beta\not\in\mathbb{N}$, let $f\in C^{\beta}([t_0,t_1];X)$. We divide into two subcases:\vspace*{0.2cm}
\begin{itemize}
\item[(a)] If there exists $m\in \mathbb{N}^*$ such that $m-1<\alpha<\beta<m$, consider the function $g:[t_0,t_1]\rightarrow X$ given by
\begin{equation}\label{eqrl1}
g(t):=f(t)-\sum_{j=0}^{m-1}\frac{f^{(j)}(t_0)}{j!}(t-t_0)^{j}.
\end{equation}

By definition, we have
\begin{equation}\label{Ccont}
  D_{t_0,t}^\beta g(t)=cD_{t_0,t}^\beta f(t),
\end{equation}
for every $t\in[t_0,t_1]$. Since $f\in C^{\beta}([t_0,t_1];X)$, we obtain that function $g\in RL^{\beta}([t_0,t_1];X)$. Thus, Theorem \ref{finalinclusionRL} ensures that $g\in RL^{\alpha}([t_0,t_1];X)$. Finally, equation \eqref{eqrl1} guarantees that $f\in C^{\alpha}([t_0,t_1];X)$. By \eqref{Ccont} and Theorem \ref{finalinclusionRL}, we obtain the continuity of the inclusion.\vspace*{0.2cm}

\item[(b)] If there exists $m\in \mathbb{N}^*$ such that $\alpha<m<\beta$, then item $(i)$ ensures that $C^{\beta}([t_0,t_1];X)$ is continuously embedded in $C^{m}([t_0,t_1];X)$, while item $(ii)$ ensures that $C^{m}([t_0,t_1];X)$ is continuously embedded in $C^{\alpha}([t_0,t_1];X)$. Thus, we deduce that $C^{\beta}([t_0,t_1];X)\subset C^{\alpha}([t_0,t_1];X)$ continuously.
\end{itemize}
\end{itemize}

Finally, the strictness follows from the fact that (see \cite{GoKiMaRo1,HaMaSa1} for details)
$$
cD^{\beta}_{t_0,t}E_\alpha\big((t-t_0)^\alpha\big)=(t-t_0)^{\alpha-\beta}E_{\alpha,1+\alpha-\beta}\big((t-t_0)^\alpha\big).
$$
In other words, $E_\alpha\big((t-t_0)^\alpha\big)$ belongs to $C^\beta\big([t_0,t_1],\mathbb{R}\big)$ for every $\beta\leq\alpha$, but it does not belong to $C^\beta\big([t_0,t_1],\mathbb{R}\big)$ when $\beta>\alpha$, since it is singular at $t=t_0$.

\end{proof}

We end this section by providing the following scheme of inclusions:

\begin{remark}  For $\alpha,\beta\in(0,\infty)$ with $[\alpha]<\beta$, the inclusions established in Theorems \ref{finalinclusionRL} and \ref{finalinclusionC} can be structured as follows:
$$
C^\beta\big([t_0,t_1];X\big)\subset C^{[\beta]-1}\big([t_0,t_1];X\big)\subset C^{[\alpha]}\big([t_0,t_1];X\big)\subset C^\alpha\big([t_0,t_1];X\big),
$$
and
$$
RL^\beta\big([t_0,t_1];X\big)\subset C_0^{[\beta]-1}\big([t_0,t_1];X\big)\subset C_0^{[\alpha]}\big([t_0,t_1];X\big)\subset RL^\alpha\big([t_0,t_1];X\big),
$$
where, for $m\in\mathbb{N}$,
\begin{equation*}
C_0^{m}([t_0,t_1];X):=\Big\{f\in C^{m}\big([t_0,t_1];X\big) \mid f^{(j)}(t_0)=0,\,\, \forall j\in\{0,1,\ldots,m\}\Big\}.
\end{equation*}

\end{remark}


\section{Properties of Fractional Continuously Differentiable Functions}\label{sec4}

In this section, we investigate the relationship between the spaces $RL^{\alpha}\big([t_0,t_1];X)$, $C^{\alpha}\big([t_0,t_1];X)$ and the Hölder spaces. Additionally, we discuss the conditions under which these spaces constitute Banach algebras.

\subsection{Fractional Continuously Differentiable Functions and H\"{o}lder Continuity}

We begin by recalling the definition of Hölder continuity.

\begin{definition} If $n\in\mathbb{N}$ and $\gamma\in(0,1)$, we define $H^{n,\gamma}(I;X)$ as the space of functions $f:I\rightarrow X$ that are $n$-times continuously differentiable and satisfy
\begin{equation*}
    \big\|f^{(n)}(t)-f^{(n)}(s)\big\|_X\leq M{|t-s|^\gamma},
\end{equation*}
for every $t,s\in[t_0,t_1]$ and some constant $M>0$. If $I$ is compact, equipping this space with the norm
\begin{equation*}
    \|f\|_{H^{n,\gamma}(I;X)} := \|f\|_{C^{n}(I;X)} + [\,f^{(n)}\,]_{H^{0,\gamma}(I;X)},
\end{equation*}
where
\begin{equation*}
    [\,f\,]_{H^{0,\gamma}(I;X)} := \sup_{t,s\in I,\,\,t\not=s} \left[\dfrac{\big\|f(t)-f(s)\big\|_X}{|t-s|^\gamma}\right],
\end{equation*}
makes $H^{n,\gamma}(I;X)$ a Banach space. Here, $f^{(j)}(t)$ denotes the standard $j$-th derivative of $f(t)$. These functions are called Hölder continuously differentiable functions of order $n$ with exponent $\gamma$. Moreover, we define
\begin{equation*}
    H^{n,\gamma}_{t_0}([t_0,t_1],X) := \big\{f\in H^{n,\gamma}([t_0,t_1],X) \mid f^{(i)}(t_0) = 0, \forall i\in\{0,\ldots,n\} \big\}.
\end{equation*}
\end{definition}

In their seminal 1928 work \cite{HaLi1}, Hardy and Littlewood established foundational theorems on the boundedness of the Riemann–Liouville fractional integral acting on $L^p$ spaces. Below, we present the result from their study that is most relevant to our analysis, using our notation.

\begin{theorem}[{\cite[Theorem 19]{HaLi1}}]
  Assume that $0<\alpha<\beta\leq 1$. Then $H^{0,\beta}_{t_0}([t_0,t_1],\mathbb{R})\subset RL^{\alpha}([t_0,t_1],\mathbb{R})$.
\end{theorem}

We provide a nontrivial proof that not only clarifies the structure of the result but also extends it to the general case where $X$ is a Banach space. This is achieved through a key characterization of the fractional derivative for Hölder continuous functions, from which the continuity of the inclusion follows directly.

\begin{theorem}\label{casenzero}
Let $f \in H^{0,\beta}_{t_0}([t_0,t_1];X)$ with $0 < \alpha < \beta < 1$. Then, the fractional derivative $D^{\alpha}_{t_0,t}f(t)$ exists for all $t \in [t_0,t_1]$ and satisfies
\begin{equation}\label{202503251405}
    D^{\alpha}_{t_0,t}f(t) = \frac{-\alpha}{\Gamma(1-\alpha)} \int_{t_0}^{t} (t-s)^{\alpha-1}[f(s)-f(t)]\,ds + \frac{(t-t_0)^{-\alpha}f(t)}{\Gamma(1-\alpha)},
\end{equation}
for every $t \in (t_0,t_1]$, while for $t = t_0$, both sides vanish. Consequently, the inclusion
\begin{equation}\label{202503231008}
H^{0,\beta}_{t_0}([t_0,t_1];X) \subsetneq RL^{\alpha}([t_0,t_1];X)
\end{equation}
holds and is continuous.

\end{theorem}

\begin{proof}
Recall that, by definition,
$$D^{\alpha}_{t_0,t}f(t) := \dfrac{d}{dt}\Big[J_{t_0,t}^{1-\alpha}f(t)\Big],$$
for every $t \in [t_0,t_1]$ where the derivative on the right-hand side exists. 

In what follows, we prove the existence of $D^{\alpha}_{t_0,t}f(t)$ directly from the limit definition of the derivative applied to $J_{t_0,t}^{1-\alpha}f(t)$. Specifically, we show that, for each $t \in (t_0,t_1)$, the left and right limits exist and coincide, and that the left-hand limit exists at $t = t_1$. In both cases, the resulting limits satisfy identity \eqref{202503251405}.

(i) At first, let us fix $t \in (t_0,t_1)$. Observe that
\begin{multline}\label{202503221036}
  \left(\dfrac{d}{dt}\right)^+J_{t_0,t}^{1-\alpha}f(t)=\lim_{\varepsilon \rightarrow 0^+} \frac{1}{\varepsilon \Gamma(1-\alpha)} \Bigg[\int_{0}^{t+\varepsilon} (t+\varepsilon-s)^{-\alpha}f(s)\,ds \\
  -\int_{0}^{t} (t-s)^{-\alpha}f(s)\,ds\Bigg] ,
\end{multline}
if the limit exists. By adding $\pm f(t+\varepsilon)$ in each integral and reorganizing, we obtain
\begin{multline}\label{202503221036}
  \left(\dfrac{d}{dt}\right)^+J_{t_0,t}^{1-\alpha}f(t) = \lim_{\varepsilon \rightarrow 0^+} \Bigg\{\underbrace{ \frac{1}{\varepsilon \Gamma(1-\alpha)} \int_{t}^{t+\varepsilon} (t+\varepsilon-s)^{-\alpha}\big[f(s) - f(t+\varepsilon)\big]\,ds }_{=:A_\varepsilon(t)}\\
  \qquad+ \underbrace{ \frac{1}{\varepsilon \Gamma(1-\alpha)} \int_{t_0}^{t} \left[ (t+\varepsilon-s)^{-\alpha} - (t-s)^{-\alpha} \right] \big[f(s) - f(t+\varepsilon)\big]\,ds }_{=:B_\varepsilon(t)} \\
    +\underbrace{ \frac{f(t+\varepsilon)}{\varepsilon \Gamma(1-\alpha)} \left[ \int_{t_0}^{t+\varepsilon} (t+\varepsilon-s)^{-\alpha}\,ds - \int_{t_0}^{t} (t-s)^{-\alpha}\,ds \right] }_{=:C_\varepsilon(t)} \Bigg\}.
\end{multline}

We now analyze the limit of each term in \eqref{202503221036}. Since $f \in H^{0,\beta}_{t_0}([t_0,t_1];X)$ and $\varepsilon>0$, we have
\begin{equation*}
  \|A_\varepsilon(t)\|_X \leq \frac{[\,f\,]_{H^{0,\beta}([t_0,t_1];X)}}{\varepsilon \Gamma(1-\alpha)} \int_{t}^{t+\varepsilon} (t+\varepsilon-s)^{\beta - \alpha}\,ds = \frac{\varepsilon^{\beta - \alpha} [\,f\,]_{H^{0,\beta}([t_0,t_1];X)}}{\Gamma(1-\alpha)(\beta - \alpha + 1)},
\end{equation*}
so that
\begin{equation}\label{derhol2}\lim_{\varepsilon \rightarrow 0^+} A_\varepsilon(t) = 0.\end{equation}

Next, observe that
\[
C_\varepsilon(t) = \left\{ \frac{(t+\varepsilon - t_0)^{1-\alpha} - (t - t_0)^{1-\alpha}}{\varepsilon \Gamma(2-\alpha)} \right\} f(t+\varepsilon),
\]
for every $\varepsilon>0$, and by L'Hôpital’s Rule,
\begin{equation}\label{derhol3}
  \lim_{\varepsilon \rightarrow 0^+} C_\varepsilon(t) = \frac{(t - t_0)^{-\alpha} f(t)}{\Gamma(1-\alpha)}.
\end{equation}

To analyze $B_\varepsilon(t)$, define
$$
\rho_\varepsilon(s) := \frac{[(t+\varepsilon-s)^{-\alpha} - (t-s)^{-\alpha}][f(s) - f(t+\varepsilon)]}{\varepsilon},
$$
for every $s\in[t_0,t)$. By L’Hôpital's Rule, for each fixed $s \in [t_0,t)$, we have
$$
\lim_{\varepsilon \rightarrow 0^+} \frac{(t+\varepsilon-s)^{-\alpha} - (t-s)^{-\alpha}}{\varepsilon} = -\alpha(t-s)^{-\alpha - 1},
$$
and therefore $\lim_{\varepsilon \rightarrow 0^+} \rho_\varepsilon(s) = -\alpha(t-s)^{-\alpha - 1}[f(s) - f(t)] := F(s),$ for each fixed $s \in [t_0,t)$.

Now observe that, since $f \in H^{0,\beta}_{t_0}([t_0,t_1];X)$, we obtain the following estimate for each fixed $s \in [t_0, t)$:
\begin{multline*}
  \|\rho_\varepsilon(s)\|_X\leq \left[\frac{[\,f\,]_{H^{0,\beta}([t_0,t_1];X)}}{\varepsilon}\right]\Big|\big[(t+\varepsilon-s)^{-\alpha}-(t-s)^{-\alpha}\big]\Big|(t+\varepsilon-s)^{\beta} \\
  =\left[\frac{[\,f\,]_{H^{0,\beta}([t_0,t_1];X)}}{\varepsilon}\right]\left\{\Big[(t-s)^{-\alpha}-(t+\varepsilon-s)^{-\alpha}\Big](t+\varepsilon-s)^{\beta}\pm (t-s)^{\beta-\alpha} \right\},
\end{multline*}
what implies that
\begin{multline*}
  \|\rho_\varepsilon(s)\|_X\leq\left[\frac{[\,f\,]_{H^{0,\beta}([t_0,t_1];X)}}{\varepsilon}\right]\left\{(t-s)^{-\alpha}\Big[(t+\varepsilon-s)^{\beta}-(t-s)^{\beta}\Big]\right.\\+ (t-s)^{\beta-\alpha} -(t+\varepsilon-s)^{\beta-\alpha}\Big\}.
\end{multline*}

Finally, since $(t-s)^{\beta-\alpha} -(t+\varepsilon-s)^{\beta-\alpha}\leq 0$ and
\begin{multline*}
  \frac{(t+\varepsilon-s)^{\beta}-(t-s)^{\beta}}{\varepsilon}=\frac{1}{\varepsilon}\int_{0}^{\varepsilon}\frac{d}{dw}(t+w-s)^\beta dw=\frac{\beta}{\varepsilon}\int_{0}^{\varepsilon}(t+w-s)^{\beta-1} dw\\
  \leq \frac{\beta}{\varepsilon}(t-s)^{\beta-1}\varepsilon= \beta(t-s)^{\beta-1},
\end{multline*}
we deduce that
 $$\|\rho_\varepsilon(s)\|_X\leq[\,f\,]_{H^{0,\beta}([t_0,t_1];X)}\beta(t-s)^{\beta-\alpha}.$$

Given that $g(s) := [\,f\,]_{H^{0,\beta}([t_0,t_1];X)}\beta (t-s)^{\beta - \alpha - 1} \in L^1(t_0,t;\mathbb{R})$ (recall that $\beta > \alpha$), we can apply the Dominated Convergence Theorem to obtain
\begin{equation}\label{derhol4}
  \lim_{\varepsilon \rightarrow 0^+}B_\varepsilon(t)=\lim_{\varepsilon \rightarrow 0^+} \int_{t_0}^{t} \rho_\varepsilon(s)\,ds = \int_{t_0}^{t} F(s)\,ds = -\alpha \int_{t_0}^{t} (t-s)^{-\alpha - 1}[f(s) - f(t)]\,ds.
\end{equation}

Combining \eqref{derhol2}, \eqref{derhol3}, and \eqref{derhol4}, we conclude that $D^{\alpha}_{t_0,t}f(t)$ exists for $t \in (t_0,t_1]$ and satisfies \eqref{202503251405}.\vspace*{0.2cm}

(ii) Now, for $t \in (t_0, t_1]$, we consider the left-hand derivative of $J_{t_0,t}^{1-\alpha}f(t)$, given by
\begin{multline*}
  \left(\dfrac{d}{dt}\right)^-J_{t_0,t}^{1-\alpha}f(t)=\lim_{\varepsilon \rightarrow 0^-} \frac{1}{\varepsilon \Gamma(1-\alpha)} \Bigg[\int_{t_0}^{t+\varepsilon} (t+\varepsilon-s)^{-\alpha}f(s)\,ds \\
  -\int_{t_0}^{t} (t-s)^{-\alpha}f(s)\,ds\Bigg],
\end{multline*}
provided the limit exists. Proceeding with analogous arguments to those used in the case of the right-hand derivative, we conclude that the left-hand derivative exists. Consequently, the RL fractional derivative of order $\alpha$ exists for every $t \in (t_0, t_1]$.\vspace*{0.2cm}

(iii) It follows directly from the H\"{o}lder regularity of $f$ that the right-hand side of identity \eqref{202503251405} vanishes at $t = t_0$. This justifies defining $D^{\alpha}_{t_0,t}f(t)\big|_{t = t_0} = 0$, thus completing the first part of the proof. Furthermore, the continuous inclusion in \eqref{202503231008} follows immediately from identity \eqref{202503251405}.

To finish the proof, we now demonstrate that the inclusion \eqref{202503231008} is strict. Consider the particular case $X = \mathbb{R}$ and define $f(t) = t^{\alpha}$. A direct computation yields $D_t^{\alpha} f(t) = \Gamma(\alpha + 1)$, showing that $f \in RL^{\alpha}([0,1];\mathbb{R})$. Besides, it is clear that $f \notin H^{0,\beta}_{t_0}([0,1];\mathbb{R})$.

\end{proof}

As a consequence of this result, we obtain the general case.

\begin{corollary}\label{holder}
   Let $0<\alpha<\beta\leq 1$ and $n\in\mathbb{N}$. Then $H^{n,\beta}_{t_0}([t_0,t_1],X)\subsetneq RL^{n+\alpha}([t_0,t_1],X)$ continuously.
\end{corollary}

\begin{proof}
  Let $f\in H^{n,\beta}_{t_0}([t_0,t_1],X)$. To show that $f\in RL^{n+\alpha}([t_0,t_1],X)$, we need only to prove that $D^{n+\alpha}_{t_0,t}f(t)\in C^0([t_0,t_1],X)$. Note that, by definition of $H^{n,\beta}_{t_0}([t_0,t_1],X)$ and item $(iii)$ of Proposition \ref{auxkilbas1},
  $$D^{n+\alpha}_{t_0,t}f(t)=\dfrac{d^{n+1}}{dt^{n+1}}J^{1-\alpha}_{t_0,t}f(t)=\dfrac{d}{dt}\Big[J^{1-\alpha}_{t_0,t}f^{(n)}(t)\Big].$$
   Since $f^{(n)}\in H^{0,\beta}_{t_0}([t_0,t_1];X)$, the result follows by reducing to the case $n=0$, which has already been fully addressed in Theorem \ref{casenzero}.

\end{proof}

We now apply the results established above to study the relationship between the spaces $H^{n,\beta}([t_0,t_1],X)$ and $C^{n+\alpha}([t_0,t_1],X)$.

\begin{corollary}\label{holder2}
  Let $0<\alpha<\beta\leq 1$ and $n\in\mathbb{N}$. Then $H^{n,\beta}([t_0,t_1],X)\subsetneq C^{n+\alpha}([t_0,t_1],X)$ continuously.
\end{corollary}

\begin{proof}
  Let $f\in H^{n,\beta}([t_0,t_1],X)$. To show that $f\in C^{n+\alpha}([t_0,t_1],X)$, we need only to prove that %
  $$cD^{n+\alpha}_{t_0,t}f(t)=D_{t_0,t}^\alpha\left[f(t)-\sum_{k=0}^{n}\frac{f^{(k)}(t_0)}{k!}\big(t-t_0\big)^{k}\right]\in C^0([t_0,t_1],X).$$
  To this end, define
  $$g(t)=f(t)-\sum_{k=0}^{n}\frac{f^{(k)}(t_0)}{k!}\big(t-t_0\big)^{k},$$
  and note that $g\in H^{n,\beta}_{t_0}([t_0,t_1],X)$.

  Since
  $cD^{n+\alpha}_{t_0,t}f(t)=D^{n+\alpha}_{t_0,t}g(t),$
  by applying Corollary \ref{holder} to function $g$ we obtain the result.
\end{proof}

\begin{remark}\label{remm}
\begin{itemize}
\item[(i)] Note that Theorem \ref{casenzero} and Corollaries \ref{holder} and \ref{holder2} do not hold in the limiting case $\alpha = \beta$. To illustrate this, let $\alpha > 0$ and suppose the domain is $[t_0, t_1]$. Define $\tilde{t} = (t_0 + t_1)/2$, and consider the function $f(t)$ given by
$f(t) = 0$ on $[t_0, \tilde{t})$, and $f(t) = 1$ on $[\tilde{t}, t_1]$.

A straightforward computation shows that
$$J^\alpha_{t_0,t} f(t) = \left\{\begin{array}{ll}0,&\textrm{ for }t \in [t_0, \tilde{t}),\vspace*{0.2cm}\\
\dfrac{1}{\Gamma(\alpha+1)}(t - \tilde{t})^{\alpha},&\textrm{ for }t \in [\tilde{t},t_1].\end{array}\right.$$

Then, by applying {\color{blue}\cite[Theorem 9]{CarFe4}}, we conclude that
$$J^\alpha_{t_0,t} f \in H^{[\alpha]-1,\alpha+1-[\alpha]}_{t_0}([t_0, t_1], \mathbb{R}).$$
However, a direct computation shows that $D^\alpha_{t_0,t} J^\alpha_{t_0,t} f(t) = f(t)$ for almost every $t \in [t_0, t_1]$, and $f$ has no representative in $C^0([t_0, t_1], \mathbb{R})$, which proves that the inclusion in Theorem \ref{casenzero} is strict. The corollaries follow directly from this conclusion.\vspace*{0.2cm}

\item[(ii)] For the case $\alpha \in (0,1)$, we also refer to the well-known counterexample originally presented by Hardy and Littlewood in \cite[Remark 5.8]{HaLi1} (cf. \cite{Ha1,KoGa1,RoStSa1}). There, they consider $\sigma \in (1,\infty)$ and define the Weierstrass function $W_\alpha:\mathbb{R} \to \mathbb{R}$ by
$$
W_\alpha(t) = \sum_{j=0}^{\infty} \sigma^{-j\alpha} \cos{(\sigma^j t)},
$$
which they observe to be H\"{o}lder continuous of order $\alpha$, yet lacking a fractional derivative of order $\alpha$ at any point.

By defining $\omega_\alpha:[t_0,t_1] \to \mathbb{R}$ as
$$
\omega_\alpha(t) := W_\alpha(t) - W_\alpha(t_0),
$$
we obtain a function satisfying $\omega_\alpha \in H^{0,\alpha}_{t_0}([t_0,t_1];X)$ that also does not admit a fractional derivative of order $\alpha$ at any point. This construction thus provides an example illustrating that $H^{0,\alpha}_{t_0}([t_0,t_1];X) \not\subset RL^{\alpha}([t_0,t_1];X)$. Nevertheless, we emphasize that this example is significantly stronger than what is strictly necessary to establish such a counterexample (cf. item $(i)$ above).

\end{itemize}
\end{remark}

To complete the analysis of the relationship between these spaces, we now consider the remaining case, which, to the best of our knowledge, has not been addressed in any known reference.

\begin{theorem}\label{HRL}
  Let $0<\alpha\leq 1$ and $n\in\mathbb{N}$. Then $RL^{n+\alpha}([t_0,t_1],X)\subsetneq H^{n,\alpha}_{t_0}([t_0,t_1],X)$ continuously.
\end{theorem}

\begin{proof} Let $f \in RL^{n+\alpha}([t_0,t_1],X)$. From item $(iii)$ of Proposition \ref{auxkilbas1} and Theorem \ref{auxkilbas2}, it follows that
$$
J^{\alpha}_{t_0,t} D^{n+\alpha}_{t_0,t} f(t) = J^{\alpha}_{t_0,t} \left\{ \dfrac{d^{n+1}}{dt^{n+1}} \left[ J_{t_0,t}^{1-\alpha} f(t) \right] \right\} = f^{(n)}(t),
$$
for every $t \in (t_0,t_1]$. Moreover, Theorem \ref{auxkilbas2} and Proposition \ref{nulo} also guarantee that
$$
J^{\alpha}_{t_0,t} D^{n+\alpha}_{t_0,t} f(t)\big|_{t=t_0} = f^{(n)}(t_0) = 0.
$$

Now, fix $x,y \in (t_0,t_1]$ (the case $x = t_0$ or $y = t_0$ is analogous) and assume, without loss of generality, that $x > y$ (the case $x = y$ is straightforward). Then, we compute:
\begin{multline}\label{202503231538}
\|f^{(n)}(x) - f^{(n)}(y)\|_X = \left\| J^{\alpha}_{t_0,x} \big[D^{n+\alpha}_{t_0,x} f(x)\big] - J^{\alpha}_{t_0,y} \big[D^{n+\alpha}_{t_0,y} f(y)\big] \right\|_X \\
\qquad\qquad\qquad= \frac{1}{\Gamma(\alpha)} \left\| \int_{t_0}^{y} \big[(x-s)^{\alpha-1} - (y-s)^{\alpha-1} \big] D^{n+\alpha}_{t_0,s} f(s) \, ds\right.\\\left. + \int_{y}^{x} (x-s)^{\alpha-1} D^{n+\alpha}_{t_0,s} f(s) \, ds \right\|_X.
\end{multline}

Since $x > y$, we have $(x-s)^{\alpha-1} < (y-s)^{\alpha-1}$ for all $s \in [t_0,y]$. Therefore, from \eqref{202503231538} we obtain:
\begin{equation}\label{202503231600}
\|f^{(n)}(x) - f^{(n)}(y)\|_X \leq \left[ \frac{\|D^{n+\alpha}_{t_0,\cdot} f\|_{C^0([t_0,t_1];X)}}{\Gamma(\alpha+1)} \right] \left[ 2(x - y)^{\alpha} + (y - t_0)^{\alpha} - (x - t_0)^{\alpha} \right].
\end{equation}

To simplify this estimate, define the auxiliary function $g:[y, t_1] \to \mathbb{R}$ by
$$
g(w) = \left\{ \begin{array}{ll}
\dfrac{2(w - y)^{\alpha} + (y - t_0)^{\alpha} - (w - t_0)^{\alpha}}{(w - y)^{\alpha}}, & \text{if } w > y, \vspace*{0.2cm}\\
2, & \text{if } w = y.
\end{array} \right.
$$
A direct computation shows that $g(w) \leq 2$ for all $w \in [y, t_1]$. Hence, from \eqref{202503231600} it follows that
\begin{equation*}
\frac{\|f^{(n)}(x) - f^{(n)}(y)\|_X}{|x - y|^{\alpha}} \leq \frac{2 \|D^{n+\alpha}_{t_0,\cdot} f\|_{C^0([t_0,t_1];X)}}{\Gamma(\alpha+1)}.
\end{equation*}

This proves that $f^{(n)}(t)$ is Hölder continuous of order $\alpha$ on $[t_0,t_1]$, and thus the inclusion holds. The continuity is an immediate consequence of the above inequality. Finally, the proper inclusion follows from the same example given at item $(i)$ of Remark \ref{remm}.
\end{proof}

\begin{corollary}\label{HRL2}
  Let $0<\alpha\leq 1$ and $n\in\mathbb{N}$. Then $C^{n+\alpha}([t_0,t_1],X)\subset H^{n,\alpha}([t_0,t_1],X)$.
\end{corollary}

\begin{proof} Let $f\in C^{n+\alpha}([t_0,t_1],X)$. As argumented in the proof of Corollary \ref{coro}, %
$$g(t):=f(t)-\sum_{k=0}^{n}\frac{f^{(k)}(t_0)}{k!}\big(t-t_0\big)^{k}\in RL^{n+\alpha}([t_0,t_1],X).$$
By Theorem \ref{HRL}, $g\in H^{n,\alpha}_{t_0}([t_0,t_1],X)$, that is, $f\in H^{n,\alpha}([t_0,t_1],X)$. Besides,      by Theorem \ref{auxkilbas2},
\begin{multline*}
  J^{\alpha}_{t_0,t}\big[cD^{\alpha}_{t_0,t}f^{(n)}(t)\big]=J^{\alpha}_{t_0,t}\Big\{D^{\alpha}_{t_0,t}\big[f^{(n)}(t)-f^{(n)}(t_0)\big]\Big\}
  \\=J^{\alpha}_{t_0,t}\left\{D^{\alpha}_{t_0,t}\left[\frac{d^n}{dt^n}g(t)\right]\right\}=J^{\alpha}_{t_0,t}\Big[D^{n+\alpha}_{t_0,t}g(t)\Big].
\end{multline*}

Now, the continuity follows from Theorem \ref{HRL}.

\end{proof}

\begin{remark}
A natural question that arises concerns the characterization of functions the belong to $\bigcap_{\alpha \in (0,1)} C^{\alpha}\big([t_0,t_1];\mathbb{R}\big)$. By the aforementioned results, Theorem \ref{HRL} and Corollaries \ref{holder}, \ref{holder2} and \ref{HRL2}, we have the identity
$$
\bigcap_{\alpha \in (\beta,\gamma)} H^{0,\alpha}([t_0,t_1];X) = \bigcap_{\alpha \in (\beta,\gamma)} C^{\alpha}([t_0,t_1];X),
$$
for every $0\leq\beta<\gamma\leq1$. An interesting consequence is that this space contains functions that do not belong to $C^1([t_0,t_1];X)$. In fact, it is already known that the vector space $\bigcap_{\alpha \in (0,1)} H^{0,\alpha}([t_0,t_1];X)$ may include functions that are not even Lipschitz continuous. The classical Weierstrass function provides an example of such behavior (cf. \cite{Ha1,HaLi1,RoStSa1}).

\end{remark}

\subsection{The spaces $RL^{\alpha}\big([t_0,t_1];X\big)$ and $C^{\alpha}\big([t_0,t_1];X\big)$ as Banach Algebras.}

In this section, we investigate the conditions under which the spaces $RL^{\alpha}\big([t_0,t_1];X\big)$ and $C^{\alpha}\big([t_0,t_1];X\big)$ are Banach algebras. To this end, we begin by recalling the following result established by Alsaedi, Ahmad, and Kirane in \cite[Lemma 1]{AlBaMk1}:
\begin{theorem}\label{eqAAK}
  Let $0<\alpha<1$. Let one of the following conditions be satisfied:
  \begin{itemize}
    \item[(i)] $u\in C^0([0, T];\mathbb{R})$ and $v\in H^{0,\beta}([0, T];\mathbb{R})$, with $\alpha<\beta\leq 1$;\vspace*{0.2cm}
    \item[(ii)] $v\in C^0([0, T];\mathbb{R})$ and $u\in H^{0,\beta}([0, T];\mathbb{R})$, with $\alpha<\beta\leq 1$;\vspace*{0.2cm}
    \item[(iii)] $v\in H^{0,\beta}([0, T];\mathbb{R})$ and $u\in H^{0,\gamma}([0, T];\mathbb{R})$, with $\alpha<\beta+\gamma$, $0<\beta,\gamma<1$.\vspace*{0.2cm}
  \end{itemize}

  Then we have
  \begin{multline}\label{leib1}
    D^{\alpha}_{0,t}(uv)(t)=u(t)D^{\alpha}_{0,t}v(t)+v(t)D^{\alpha}_{0,t}u(t)\\
    -\frac{\alpha}{\Gamma{(1-\alpha)}}\int_{0}^{t}{\frac{\big[u(s)-u(t)\big]\big[v(s)-v(t)\big]}{(t-s)^{\alpha+1}}ds}-\frac{u(t)v(t)}{\Gamma(1-\alpha)t^\alpha},
  \end{multline}
  for every $t\in(0,T]$.
\end{theorem}

We would like to point out that the hypotheses in Theorem \ref{eqAAK}, as currently stated, might not be sufficient to guarantee the validity of identity \eqref{leib1} in all cases. Specific counterexamples indicate that additional conditions may be necessary to ensure the result holds. Below, we present illustrative examples corresponding to each of the items listed in the theorem.
\begin{itemize}
\item[(i)] Consider $u(t) = \omega_\alpha(t)$, where $\omega_\alpha(t)$ is the function defined in item $(ii)$ of Remark \ref{remm}, and let $v(t) = 1$. If item $(i)$ of Theorem \ref{eqAAK} were valid as stated, we could apply it to deduce that their product, which is simply $\omega_\alpha(t)$, possesses a RL fractional derivative. However, we already know that $\omega_\alpha(t)$ does not admit such a derivative.\vspace*{0.2cm}
\item[(ii)] Analogous to the above commentary, just take $v(t)=\omega_\alpha(t)$ and $u(t)=1$.\vspace*{0.2cm}
\item[(iii)] Here there would be no counter example if $\beta,\gamma>\alpha$ (see Theorem \ref{casenzero}), since in this case $D_t^\alpha u\in C^{0}([0, T];\mathbb{R})$ and $D_t^\alpha v\in C^{0}([0, T];\mathbb{R})$. If $\beta\leq\alpha$ or $\gamma\leq\alpha$, then we could repeat the arguments from items $(i)$ or $(ii)$, depending on the situation to obtain the contradiction.
\end{itemize}

Given the preceding discussion, we understand that the statement of Theorem \ref{eqAAK} requires the addition of certain assumptions to ensure the validity of identity \eqref{leib1}. A possible minimal set of complementary hypotheses that guarantees the result is the following:
\begin{itemize}
\item[(i)] $u\in C^0([0, T];\mathbb{R})$, the RL fractional derivative of order $\alpha$ of $u$ exists at $t\in(0,T]$, and $v\in H^{0,\beta}([0, T];\mathbb{R})$, with $\alpha<\beta\leq 1$;\vspace*{0.2cm}
\item[(ii)] $v\in C^0([0, T];\mathbb{R})$, the RL fractional derivative of order $\alpha$ of $v$  exists at $t\in(0,T]$, and $u\in H^{0,\beta}([0, T];\mathbb{R})$, with $\alpha<\beta\leq 1$;\vspace*{0.2cm}
\item[(iii)] $v\in H^{0,\beta}([0, T];\mathbb{R})$, $u\in H^{0,\gamma}([0, T];\mathbb{R})$, with $\alpha<\beta+\gamma$, $0<\beta,\gamma<1$, and the RL fractional derivatives of order $\alpha$ of both $u$ and $v$ exists at $t\in(0,T]$.\vspace*{0.2cm}
\end{itemize}

In addition to the adjustments proposed above to make Theorem \ref{eqAAK} suitable for the applications we develop in the final part of this work, we also consider it important to point out some additional observations about identity \eqref{leib1}.

\begin{remark}
  \begin{enumerate}
    \item[(i)] In the case of the Caputo derivative, we have:
    \begin{multline*}
      cD^{\alpha}_{0,t}(uv)(t) = u(t)\,cD^{\alpha}_{0,t}v(t) + v(t)\,cD^{\alpha}_{0,t}u(t) \\
      - \frac{\alpha}{\Gamma(1-\alpha)} \int_{0}^{t} \frac{\big[u(s) - u(t)\big]\big[v(s) - v(t)\big]}{(t-s)^{\alpha+1}}\,ds
      - \frac{\big[u(t) - u(0)\big]\big[v(t) - v(0)\big]}{\Gamma(1-\alpha)t^\alpha}.
    \end{multline*}

    \item[(ii)] All these identities continue to hold when the fractional derivatives are considered with respect to a general initial point $t_0 \neq 0$, over the interval $[t_0, t_1]$.
  \end{enumerate}
\end{remark}

It is a classical result that the product of two differentiable functions is itself differentiable. To conclude our work, we present the fractional counterpart of this property. Specifically, we prove that the product of two functions, each possessing a continuous RL (or Caputo) fractional derivative, also admits a continuous fractional derivative of the same type.

\begin{theorem}\label{fcoro}
Let $0<\alpha<1$, then $RL^\alpha([t_0,t_1], \mathbb{R})$ and $C^\alpha([t_0,t_1], \mathbb{R})$ are Banach algebras.
\end{theorem}

\begin{proof} Since the two cases share analogous proofs, we focus only on proving that $RL^\alpha([t_0,t_1], \mathbb{R})$ is a Banach algebra. Let $f, g \in RL^\alpha([t_0,t_1], \mathbb{R})$. As $fg \in C^0([t_0,t_1], \mathbb{R})$, it remains to verify that $D^{\alpha}_{t_0,t}(fg) \in C^0([t_0,t_1], \mathbb{R})$. By Theorem \ref{HRL}, we have that $f, g \in H^{0,\alpha}([t_0, t_1], \mathbb{R})$, and thus we may apply item $(iii)$ of the refined version of Theorem \ref{eqAAK}. This ensures that $D^{\alpha}_{t_0,t}(f(t)g(t))$ exists and satisfies the identity \eqref{leib1} for every $t\in(t_0,t_1]$. Since $f(t_0) = g(t_0) = 0$ and $f,g \in H^{0,\alpha}([t_0, t_1], \mathbb{R})$, we conclude that
$$\lim_{t \to t_0} D_{t_0,t}^\alpha(f(t)g(t)) = 0,$$
proving that the right-hand side of \eqref{leib1} is continuous, which implies that $D^{\alpha}_{t_0,t}(fg) \in C^0([t_0,t_1], \mathbb{R})$.

\end{proof}

Finally, it follows directly from the definition of the vector-valued spaces $RL^\alpha([t_0,t_1], X)$ and  $C^\alpha([t_0,t_1], X)$, from Remark \ref{eqre02111} and Theorem \ref{fcoro}, that the following result also holds true.

\begin{corollary}
  If $\alpha>0$ and $X$ is an Banach algebra, then the spaces $RL^\alpha([t_0,t_1], X)$ and  $C^\alpha([t_0,t_1], X)$ are Banach algebras.
\end{corollary}

\section*{Acknowledgement}
 The authors were supported by Fundação de Amparo à Pesquisa e Inovação do Espírito Santo (FAPES) under grant number T.O. 951/2023.

\end{document}